\documentclass[10pt,reqno]{amsart}
\usepackage{amsmath,amscd}
\usepackage{url}

\theoremstyle{plain}
   \newtheorem{theorem}{Theorem}[section]
   \newtheorem{proposition}[theorem]{Proposition}
   \newtheorem{lemma}[theorem]{Lemma}
   
   \newtheorem{conjecture}[theorem]{Conjecture}
   \newtheorem{problem}[theorem]{Problem}
\theoremstyle{definition}

\newtheorem{example}[theorem]{Example} 

\theoremstyle{remark}
 \newtheorem{remark}[theorem]{Remark}

\author[P.~Br\"and\'en]{Petter Br\"and\'en}
\thanks{Supported by the G\"oran Gustafsson Foundation.}
\thanks{To appear in J. Reine Angew. Math. (Crelle's journal)} 
       \address{Department of Mathematics, 
Stockholm University, 
SE-106 91 Stockholm, Sweden}
\email{pbranden@math.su.se}


\def\kk{\kern.2ex\mbox{\raise.5ex\hbox{{\rule{.35em}{.12ex}}}}\kern.2ex}

\newcommand{\NN}{\mathbb{N}}
\newcommand{\zz}{\mathbf{z}}

\newcommand{\LP}{\mathcal{L{\kk}P}}
\newcommand{\PP}{\mathcal{P}^+}

\newcommand{\HH}{\mathcal{H}}
\newcommand{\LL}{\mathcal{L}}
\newcommand{\TT}{\mathcal{T}}

\newcommand{\A}{\mathcal{A}}

\newcommand{\RR}{\mathbb{R}}
\newcommand{\CC}{\mathbb{C}}

\renewcommand{\Re}{{\rm Re}}

\def\newop#1{\expandafter\def\csname #1\endcsname{\mathop{\rm
#1}\nolimits}}

\newop{diag}
\newop{supp}
\newop{per}
\newop{JP}
\newop{Sym}
\newop{St}
\newop{glb}
\newop{Trp}
\newop{Orb}

\dedicatory{Dedicated to the memory of Julius Borcea}

\begin{document}
\title[Iterated sequences and the geometry of zeros]
{Iterated sequences and the geometry of zeros}
\begin{abstract}
We study the effect on the zeros of generating functions of sequences under certain non-linear transformations.  Characterizations of P\'olya--Schur type are given of the transformations that preserve the property of having only real and 
non-positive zeros. 
In particular, if a polynomial $a_0+a_1z +\cdots+a_nz^n$ has only real and non-positive zeros, 
then so does the polynomial 
$a_0^2+ (a_1^2-a_0a_2)z+\cdots+ (a_{n-1}^2-a_{n-2}a_n)z^{n-1}+a_n^2z^n$. 
This confirms a conjecture of Fisk, McNamara--Sagan and Stanley, respectively. A consequence is that if a polynomial has only real and non-positive zeros, then its Taylor coefficients form an infinitely log-concave sequence.  We extend the results to transcendental entire functions in the Laguerre--P\'olya class, and discuss the consequences to problems on iterated Tur\'an inequalities, studied by Craven and Csordas. Finally, we propose a new approach to a conjecture of Boros and Moll. 
\end{abstract}
\maketitle
\tableofcontents

\section{Introduction}
Let $\Phi$ be a transformation of sequences of real numbers, and let  $\{a_k\}$ be a real sequence. We are interested in when the iterates $\Phi^i(\{a_k\})$, for $i\in \NN$, are non-negative. Such questions appear in the theory of entire functions \cite{CC1,CC2}, and recently in the theory of special functions \cite{BM,Fi,KP} and combinatorics \cite{MS,St}. It has been made evident that the zero set of the generating function of $\{a_k\}$ plays a prominent role in such questions. One purpose of this paper is to make this correspondence explicit. 

 Let $\{a_k\}=\{a_k\}_{k=0}^n$, where $n \in \NN\cup \{\infty\}$,  be a sequence of real numbers. The sequence is {\em log-concave} if $a_k^2-a_{k-1}a_{k+1} \geq 0$, for all $1 \leq k \leq n-1$. Define an operator on sequences by $\LL(\{a_k\}) = \{b_k\}_{k=0}^n$, where $b_k=a_k^2-a_{k-1}a_{k+1}$ for all $0 \leq k \leq n$, and  $a_{-1}=a_{n+1}=0$. Hence, $\{a_k\}$ is log-concave if and only if $\LL(\{a_k\})$ is non-negative. The sequence is $i$-\emph{fold log-concave} if the $i$th iterate,  $\LL^i(\{a_k\})$, is non-negative, and 
 {\em infinitely log-concave} if it is $i$-fold log-concave  for all $i \in \NN$. Boros and Moll \cite{BM} conjectured  
 that the sequence of binomial numbers, $\{ \binom n k \}_{k=0}^n$, is infinitely log-concave for each $n \in \NN$. If the polynomial 
 $
 \sum_{k=0}^n a_k z^k
 $
 has only real and non-positive zeros, then it follows that the sequence $\{a_k\}$ is log-concave. Motivated by this fact and Boros and Moll's conjecture on binomial numbers, Stanley \cite{St}, McNamara--Sagan \cite{MS} and Fisk \cite{Fi}, independently  made the following conjecture. 
\begin{conjecture}\label{conMSS}
Suppose that the polynomial $\sum_{k=0}^n a_k z^k$ has only real and negative zeros. Then so does the polynomial 
$$
\sum_{k=0}^n (a_k^2-a_{k-1}a_{k+1})z^k, \quad \mbox{ where } a_{-1}=a_{n+1}=0.
$$
In particular, the sequence $\{a_k\}_{k=0}^n$ is infinitely log-concave. 
\end{conjecture}
It should be mentioned that similar questions were raised already in \cite{CC1,CC2}, see Section \ref{fin}.
In Section \ref{SGWS} we prove 
Conjecture \ref{conMSS}.  
However, we take a general approach and study a large class of transformations of sequences. Let $\alpha =\{\alpha_k\}_{k=0}^\infty$ be a fixed sequence of complex numbers and define two sequences   $\LL_\alpha^E ( \{a_k\}_{k=0}^n)=\{b_k(\alpha)\}_{k=0}^n$ and $\LL_\alpha^O ( \{a_k\}_{k=0}^n)=\{c_k(\alpha)\}_{k=0}^n$, where 
$$
b_k(\alpha)= \sum_{j=0}^\infty \alpha_j a_{k-j}a_{k+j} \quad \mbox{ and } \quad  c_k(\alpha)= \sum_{j=0}^\infty \alpha_j a_{k-j}a_{k+1+j}, 
$$
 and $a_j=0$ if $j \not \in \{0, \ldots, n\}$. In Theorems \ref{trans-U} and \ref{trans-V} we characterize the sequences $\alpha$, for which $\LL^E$ (or $\LL^O$) preserves the property of having generating polynomial with only real and non-positive zeros. The characterization is of \emph{P\'olya--Schur type}; that is, $\LL^E_\alpha$ (or $\LL^O_\alpha$) has the desired properties if and only if the generating function of $\LL_\alpha^E(\{1/k!\})$ (or $\LL_\alpha^O(\{1/k!\})$) is an entire function that can be approximated, uniformly on compact subsets of $\CC$, by polynomials with only negative zeros. Similar characterizations of classes of transformations were given in \cite{PS} and \cite{BB}.  The fundamental difficulty in our setting is that in \cite{BB,PS},  the transformations in question are linear, whereas the transformations that we consider are not.  This potential problem is overcome by a symmetric function identity (Theorem \ref{gen-id}) that linearizes the problem. 

In Section \ref{fin} we propose a new approach to the original conjecture (see Conjecture \ref{orig-con})  of Boros and Moll \cite{BM}. We state a conjecture that would imply $3$-fold log-concavity of the sequences in question.

\section{Symmetric function identities}\label{sym}
Let $\{e_k(\zz)\}_{k=0}^n$ denote the elementary symmetric functions in the variables $\zz=(z_1, \ldots, z_n)$, and set $e_k(\zz)=0$ for $k \notin \{0,\ldots,n\}$. If $\mu=\{\mu_k\}_{k=0}^\infty$ is a sequence of complex numbers, we define a symmetric function by 
$$
W_{\mu,n}(\zz)= \sum_{i\leq j}\mu_{j-i}e_i(\zz)e_j(\zz).
$$

\begin{theorem}\label{gen-id}
Let $\mu = \{\mu_k\}_{k=0}^\infty$ be a sequence of complex numbers, and let  
\begin{equation}\label{gammas}
\gamma_k = \sum_{j=0}^{\lfloor k/2\rfloor} \binom  k j \mu_{k-2j}, \quad \mbox{ for } k\in \NN.
\end{equation}
Then 
\begin{equation}\label{el-exp}
W_{\mu,n}(\zz)= e_n(\zz)\sum_{k=0}^n \gamma_k e_{n-k}\left(\zz+\frac 1 \zz\right), 
\end{equation}
where $1/\zz=(1/z_1, \ldots, 1/z_n)$. 

\end{theorem}

\begin{proof}
By linearity it is enough to prove the theorem for the case when there is a number $m \in \NN$ such that $\mu_m=1$, and $\mu_k=0$ for each $k\neq m$. 

Following \cite{CR,NU}, for $k,r, n \in \NN$, define a symmetric function $\sigma_k^r(\zz)$ by 
$$
\sigma_k^r(\zz) = \sum_{\alpha=(\alpha_1, \ldots, \alpha_n)} z_1^{\alpha_1}\cdots z_n^{\alpha_n}, 
$$
where the summation is over all $\alpha \in \{0,1,2\}^n$ such that $\alpha_1 + \cdots + \alpha_n = k$, and $|\{i : \alpha_i = 2\}|= r$. By a simple counting argument, see \cite{CR,NU},
\begin{equation}\label{prodform}
e_i(\zz)e_j(\zz)= \sum_{r}\binom {i-r +j-r}{i-r} \sigma_{i+j}^r(\zz), 
\end{equation} 
and thus 
\begin{equation}\label{prodform2}
e_{i}(\zz)e_{i+m}(\zz) = \sum_{j} \binom {2j+m} {j} \sigma_{m+2i}^{i-j}(\zz).
\end{equation}

From the definition of $\sigma_{m+2i}^{i-j}(\zz)$ we see that 
$$
\sigma_{m+2i}^{i-j}(\zz)= \sum_{|S|=m+2j}\zz^S e_{i-j}(z_t^2: t \not \in S), 
$$
where $\zz^S=\prod_{s \in S}z_s$. Summing over all $i$ in the equation above yields
$$
\sum_{|S|=m+2j}\zz^S \prod_{t \not \in S}(1+z_t^2)= {e_n(\zz)} e_{n-m-2j}\left(z_1+\frac 1 {z_1}, \ldots, z_n+\frac 1 {z_n}\right). 
$$
Equation \eqref{el-exp}, for our choice of $\mu$,  now follows from \eqref{prodform2} when summing over all $j$. 
\end{proof}

We pause here to sketch an alternative combinatorial proof of the important case of \eqref{el-exp} when 
$\mu=\{1,0,-1,0,0,\ldots\}$. For undefined symmetric function terminology we refer to \cite[Chapter 7]{St2}. For our particular choice of $\mu$, we want to prove the identity
\begin{equation}\label{beauty}
\sum_{k=0}^n (e_k(\zz)^2-e_{k-1}(\zz)e_{k+1}(\zz))= e_n(\zz)\sum_{k=0}^{\lfloor n/2 \rfloor}  C_k e_{n-2k}\left(\zz+\frac 1 \zz\right),
\end{equation}
where $C_k=\binom {2k} k /(k+1)$ is a \emph{Catalan number}, see \cite[Exercise 6.19]{St2}.  We may rewrite \eqref{beauty} as
\begin{equation}\label{cat-type}
\sum_{k=0}^n (e_k(\zz)^2-e_{k-1}(\zz)e_{k+1}(\zz))= \sum_{k=0}^{\lfloor n/2 \rfloor} C_k \sum_{|S|=2k}\zz^S\prod_{j \notin S}(1+z_j^2). 
\end{equation}
The polynomial $e_k(\zz)^2-e_{k-1}(\zz)e_{k+1}(\zz)$ is the Schur-function $s_{2^k}(\zz)$, where $2^k=(2,2,\ldots,2)$.  By the combinatorial definition of the Schur-function, the left hand side of \eqref{cat-type} is the generating polynomial of all semi-standard Young tableaux  with entries in $\{1, \ldots, n\}$, that are  of shape $2^k$ for some $k \in \NN$. Call this set $\A_n$. Given $T \in \A_n$, let $S$ be the set of entries which occur only ones in $T$. By deleting the remaining entries we obtain a standard Young tableau  of shape $2^{k}$, where $2k=|S|$. There are exactly $C_k$ standard Young tableaux of shape 
$2^k$ with set of entries $S$, see e.g. \cite[Exercise 6.19.ww]{St2}. The original semi-standard Young tableau is then determined by the set of duplicates. This explains the right hand side 
of \eqref{cat-type}.  

\section{Grace--Walsh--Szeg\H{o} type theorems and a proof of Conjecture \ref{conMSS}}\label{SGWS}
The Grace--Walsh--Szeg\H{o} Theorem  is undoubtably one of the most useful theorems governing the location of zeros of polynomials,  see \cite{RS}. A {\em circular region} is a proper subset of the complex plane that is bounded by either a circle or a straight line, and is either open or closed. A polynomial is {\em multi-affine} provided that each 
variable occurs at most to the first power.   
\begin{theorem}[Grace--Walsh--Szeg\H{o}]\label{GWS}
Let $f \in \CC[z_1, \ldots, z_n]$ be a multi-affine and symmetric polynomial, and let $K$ be a circular region. Assume that either $K$ is convex or that the degree of $f$ is $n$. For any $\zeta_1, \ldots, \zeta_n \in K$ there is a $\zeta \in K$ such that 
$
f(\zeta_1, \ldots, \zeta_n)= f(\zeta, \ldots, \zeta). 
$
\end{theorem}
We are now in a position to prove Conjecture \ref{conMSS}. 

\begin{proof}[Proof of Conjecture \ref{conMSS}]
Let $P(z)= \sum_{k=0}^n a_kz^k= \prod_{k=0}^n(1+\rho_k z)$, where $\rho_k >0$ for all $1 \leq k \leq n$, and let 
$$
Q(z)= \sum_{k=0}^n (a_k^2-a_{k-1}a_{k+1})z^k.
$$
Suppose that there is a $\zeta \in \CC$, with $\zeta \notin \{x \in \RR: x \leq 0\}$, for which $Q(\zeta)=0$. We may write $\zeta$ as 
$\zeta = \xi^2$, where $\Re(\xi)>0$. By \eqref{beauty}, 
$$
0=Q(\zeta)=  a_n\xi^n \sum_{k=0}^{\lfloor n/2 \rfloor} C_k e_{n-2k}\left(\rho_1\xi+\frac 1 {\rho_1\xi}, \ldots, \rho_n\xi+\frac 1  {\rho_n\xi}\right), 
$$
where $C_k= \binom {2k} k/(k+1)$. Since $\Re(\rho_j \xi + 1/(\rho_j \xi))>0$ for all $1\leq j \leq n$, the Grace--Walsh--Szeg\H{o} Theorem provides a $\eta \in \CC$, with $\Re(\eta)>0$, such that 
$$
0=\sum_{k=0}^{\lfloor n/2 \rfloor} C_k e_{n-2k}\left(\eta, \ldots, \eta\right)= \sum_{k=0}^{\lfloor n/2 \rfloor} C_k\binom n {2k} \eta^{n-2k}=: \eta^n p_n\left(\frac 1 {\eta^2}\right).  
$$
Since $\Re(\eta)>0$, we have $1/\eta^2 \in \CC \setminus \{ x \in \RR: x\leq  0\}$. Hence, the desired contradiction follows if we can prove that all the zeros of $p_n(z)$ are real and negative. This follows from the identity 
\begin{eqnarray*}
\sum_{k=0}^{\lfloor n/2 \rfloor} C_k \binom n {2k}z^{k}(1+z)^{n-2k}&=&
\sum_{k=0}^n \frac 1 {n+1} \binom {n+1} k \binom {n+1} {k+1} z^k\\ 
&=& \frac 1 {n+1} (1-z)^nP_n^{(1,1)}\left(\frac {1+z}{1-z} \right), 
\end{eqnarray*}
where $\{P_n^{(1,1)}(z)\}_n$ are \emph{Jacobi polynomials}, see \cite[p.~254]{Ra}.  The zeros of the Jacobi polynomials $\{P_n^{(1,1)}(z)\}_n$ are located in the interval $(-1,1)$. Note that the first identity in the equation above follows 
immediately from \eqref{beauty}.
\end{proof}

Now that Conjecture \ref{conMSS} is established we shall see how the ideas in the proof can be extended considerably. 

 If $\mu$ is a sequence of complex numbers, define a (non-linear) operator, $T_\mu : \CC[z] \rightarrow \CC[z]$, by 
\begin{equation}\label{tmu}
T_\mu\left(\sum_{k=0}^na_kz^k\right)= \sum_{i \leq j} \mu_{j-i}a_ia_jz^{i+j}. 
\end{equation}

Define polynomials, $P_{\mu,n}(z)$, for $n \in \NN$, by 
$$
P_{\mu,n}(z)= \sum_{k=0}^n \gamma_{k}\binom n kz^{n-k}= \sum_{j,k}\binom n k \binom  k j \mu_{k-2j} z^{n-k}.
$$

A complex polynomial $F(z_1,\ldots, z_n)$ is \emph{weakly Hurwitz stable} if 
$F(z_1,\ldots, z_n) \neq 0$ whenever $\Re(z_j)>0$ for all $1\leq j\leq n$. The following theorem can be seen as a Grace--Walsh--Szeg\H{o} theorem for certain non-multi-affine polynomials. 
\begin{theorem}\label{GWS2}
Let $\mu$ be a sequence of complex numbers. The following are equivalent. 
\begin{itemize}
\item[(i)] $W_{\mu,n}(\zz)$ is weakly Hurwitz stable; 
\item[(ii)] For all polynomials $P(z)$ of degree at most $n$, with only real and non-positive zeros,  the polynomial 
$T_\mu(P(z))$ is either identically zero or weakly Hurwitz stable;
\item[(iii)] $T_\mu\left((1+z)^n\right)=W_{\mu,n}(z,\ldots, z)$  is weakly Hurwitz stable;
\item[(iv)] The polynomial $P_{\mu,n}(z)$ is weakly Hurwitz stable.
\end{itemize}
\end{theorem}
\begin{proof}
Suppose that $W_{\mu,n}(\zz)$ is weakly Hurwitz stable, and that $P(z)$ is a real polynomial of degree at most $n$ with only real and non-positive zeros. By Hurwitz' theorem on the continuity of zeros, see e.g. \cite[Theorem 1.3.8]{RS}, we may assume that 
$P(z)= \prod_{j=1}^n(1+\rho_jz)$, where $\rho_j>0$ for all $1\leq j \leq n$. Suppose that $\Re(\zeta)>0$. Then $\Re(\rho_j\zeta)>0$, for all $1\leq j \leq n$.  Hence 
$$
W_{\mu,n}(\rho_1\zeta, \ldots, \rho_n \zeta) = T_\mu(P(\zeta))\neq 0, 
$$
which proves (i) $\Rightarrow$ (ii). 

The implication (ii)  $\Rightarrow$ (iii) is obvious. Clearly,  by \eqref{el-exp}, 
$$
W_{\mu,n}(z,\ldots, z)=z^n P_{\mu,n}\left(z+\frac 1 z\right). 
$$
Hence, the equivalence of  (iii) and (iv) follows from the set identity 
$$\{ z+ 1/z : z \in \CC \mbox{ and } \Re(z)>0\} = \{z \in \CC: \Re(z)>0\}.$$
Now, suppose that $W_{\mu,n}(\zeta_1,\ldots, \zeta_n)=0$, where $\Re(\zeta_j)>0$ for all $1\leq j \leq n$.  Then, by \eqref{el-exp}, 
$$
\sum_{k=0}^n \gamma_k e_{n-k}\left(\zeta_1+ \frac 1 {\zeta_1}, \ldots,  \zeta_n+ \frac 1 {\zeta_n}\right)=0.
$$
Since $\Re(\zeta_j+ 1/\zeta_j)>0$, for all $1\leq j\leq n$, the Grace--Walsh--Szeg\H{o} theorem provides a number $\xi \in \CC$, with $\Re(\xi)>0$, such that 
$$
0= \sum_{k=0}^n \gamma_k e_{n-k}\left(\xi,\ldots,\xi\right)= P_{\mu,n}(\xi). 
$$
This verifies (iv) $\Rightarrow$ (i). 
\end{proof}

\section{Algebraic P\'olya--Schur characterizations of transformations} 
Let us turn to the cases when all the non-zero $\mu_i$'s have the same parity. Let $\alpha =\{\alpha_k\}_{k=0}^\infty$ be a fixed sequence of complex numbers and define two sequences   $\LL_\alpha^E ( \{a_k\}_{k=0}^n)=\{b_k(\alpha)\}_{k=0}^n$ and $\LL_\alpha^O ( \{a_k\}_{k=0}^n)=\{c_k(\alpha)\}_{k=0}^n$, where 
$$
b_k(\alpha)= \sum_{j=0}^\infty \alpha_j a_{k-j}a_{k+j} \quad \mbox{ and } \quad  c_k(\alpha)= \sum_{j=0}^\infty \alpha_j a_{k-j}a_{k+1+j}, 
$$
 and $a_j=0$ if $j \not \in \{0,1, \ldots, n\}$. Define also two  
\emph{non-linear} operators on polynomials, $U_\alpha, V_\alpha : \CC[z] \rightarrow \CC[z]$, by 
$$
U_\alpha \left( \sum_{k=0}^n a_kz^k\right) = \sum_{k=0}^n b_k(\alpha) z^k \quad \mbox{ and } \quad
V_\alpha \left( \sum_{k=0}^n a_kz^k\right) = \sum_{k=0}^n c_k(\alpha) z^k.
$$
We want to characterize the real sequences $\alpha$ for which $U_\alpha$ (or $V_\alpha$) send polynomials with only real and non-positive zeros to polynomials of the same kind. 

If $P(z)= \sum_{k=0}^na_k z^k$, let 
$$
P^E(z)= \sum_{k=0}^{\lfloor n/2\rfloor} a_{2k}z^k \quad \mbox{ and } \quad P^O(z)= \sum_{k=0}^{\lfloor (n-1)/2\rfloor} a_{2k+1}z^k. 
$$
The next theorem is a version of the classical Hermite--Biehler theorem, see e.g. \cite[p. 197]{RS}.  
\begin{theorem}[Hermite--Biehler]
Let $P(z)= P^E(z^2)+ zP^O(z^2) \in \RR[z]$.  Then $P(z)$ is weakly Hurwitz stable if and only if 
all non-zero coefficients of $P$ have the same sign, and 
\begin{itemize}
\item $P^E(z) \equiv 0$, and $P^O(z)$ has only real and non-positive zeros, or
\item $P^O(z) \equiv 0$, and $P^E(z)$ has only real and non-positive zeros, or
\item $P^E(z)P^O(z) \not \equiv 0$, and $P^E(z)$ and $P^O(z)$ have real and non-positive zeros which are interlacing in the following sense. If 
$z'_m\leq \cdots \leq z'_1$ and $z_\ell\leq \cdots \leq z_1$ are the zeros of $P^E(z)$ and $P^O(z)$, respectively, then 
$$
\cdots     \leq z_3\leq z'_2 \leq z_2\leq z'_1\leq z_1. 
$$
\end{itemize} 
\end{theorem}
Given a sequence $\mu$, we define two auxiliary operators, $T_\mu^E, T_\mu^O : \CC[z] \rightarrow \CC[z]$, by 
$$
T_\mu^E(P(z))= T_\mu(P(z))^E \quad \mbox{ and } \quad T_\mu^O(P(z))= T_\mu(P(z))^O. 
$$

Let $\PP_n$ denote the set of all polynomials of degree at most $n$ with only real and non-positive zeros, and let $\PP=\bigcup_{n=0}^\infty \PP_n$.
\begin{theorem}\label{alg-U}
Let $\alpha=\{\alpha_k\}_{k=0}^\infty$ be a sequence of real numbers, and let $n\in \NN$. The following are equivalent. 
\begin{itemize}
\item[(i)] $U_\alpha(\PP_n) \subseteq \PP_n\cup\{0\}$;
\item[(ii)]  $U_\alpha\left((1+z)^n\right) \in \PP_n\cup\{0\}$;
\item[(iii)]  
$$
\sum_{k=0}^{\lfloor n/2 \rfloor} \left( \sum_{j=0}^k \frac {\alpha_j}{(k+j)!(k-j)!} \right)\frac{z^k}{(n-2k)!} \in \PP_n\cup\{0\}.
$$
\end{itemize}
\end{theorem}
\begin{proof}
Let $\mu=\{\alpha_0,0, \alpha_1,0, \alpha_2,\ldots\}$, and consider the operator $T_\mu$ given by \eqref{tmu}. Then $U_\alpha= T_\mu^E$. By the Hermite--Biehler theorem, for each $P \in \PP_n$,  
$$
T_\mu(P) \mbox{ is weakly Hurwitz stable  if and only if }  U_\alpha(P) \in \PP_n. 
$$
Now, 
$$
P_{\mu,n}(z) = n!\sum_{k=0}^{\lfloor n/2 \rfloor} \left( \sum_{j=0}^k \frac {\alpha_j}{(k+j)!(k-j)!} \right)\frac{z^{n-2k}}{(n-2k)!}, 
$$
so $P_{\mu,n}(z)$ is weakly Hurwitz stable or identically zero if and only if (iii) holds. The theorem follows from Theorem \ref{GWS2}. 
\end{proof}
\begin{example}
Let us use Theorem \ref{alg-U} to give a second proof of Conjecture \ref{conMSS}. In this situation $\alpha=\{1,-1,0, 0, \ldots\}$, and 
$$
U_\alpha((1+z)^n)= \sum_{k=0}^n\left( \binom {n} k^2 - \binom {n} {k-1}\binom {n} {k+1}\right)z^k=\sum_{k=0}^n \frac 1 {n+1} \binom {n+1} k \binom {n+1} {k+1} z^k. 
$$
These polynomials are known as the \emph{Narayana polynomials}. There are numerous proofs that the Narayana polynomials have only real zeros. The simplest is probably based on  the Mal\'o Theorem, see e.g. \cite[Theorem 2.4]{CC3}. 
\end{example}

The corresponding theorem for $V_\alpha$ reads as follows. 
\begin{theorem}\label{alg-V}
Let $\alpha=\{\alpha_k\}_{k=0}^\infty$ be a sequence of real numbers, and let $n\in \NN$. The following are equivalent. 
\begin{itemize}
\item[(i)] $V_\alpha(\PP_n) \subseteq \PP_n\cup\{0\}$;
\item[(ii)]  $V_\alpha\left((1+z)^n\right) \in \PP_n\cup\{0\}$;
\item[(iii)]  
$$
\sum_{k=0}^{\lfloor (n-1)/2 \rfloor} \left( \sum_{j=0}^k \frac {\alpha_j}{(k+1+j)!(k-j)!} \right)\frac{z^k}{(n-2k-1)!} \in \PP_n\cup\{0\}.
$$
\end{itemize}
\end{theorem}
\begin{proof}
Consider $\mu=\{0,\alpha_0, 0, \alpha_1, 0,\ldots\}$. The proof proceeds just as  the proof of Theorem \ref{alg-U}, since $V_\alpha= T_\mu^O$.
\end{proof}

\section{Transcendental  P\'olya--Schur characterizations of  transformations}
In this section we provide transcendental characterizations of various transformations. The following spaces of entire functions are relevant for our purposes. 
\begin{itemize}
\item $\HH(\CC)$ is the set of entire functions that are limits, uniformly on compact subsets of $\CC$, of univariate polynomials that have zeros only  in the closed left half-plane;
\item $\HH(\RR)$ is the space of entire functions in $\HH(\CC)$ with real coefficients;
\item The \emph{Laguerre--P\'olya class}, $\LP$, of entire functions consists of all entire functions that are limits, uniformly on compact subsets of $\CC$,  of real polynomials with only real zeros. A function $\phi$ is in $\LP$ if and only if it can be expressed in the form 
$$
\phi(z)= C z^n e^{-az^2 +bz}\prod_{j=0}^\infty (1+\rho_jz)e^{-\rho_j z},
$$
where $n \in \NN$, $a,b,c \in \RR$, $a \geq 0$, and $\{\rho_j\}_{j=0}^\infty \subset \RR$ satisfies $\sum_{j=0}^\infty \rho_j^2< \infty$, see \cite[Chapter VIII]{Le};
\item $\LP^+$ consists of those functions in the Laguerre--P\'olya class that have non-negative Taylor coefficients. A function $\phi$ is in $\LP^+$ if and only it can be expressed as 
$$
\phi(z) = Cz^M e^{az}\prod_{j=0}^\infty (1+\rho_jz), 
$$
where $a, C\geq 0, M \in \NN$ and $\sum_{j=0}^\infty \rho_j < \infty$, see \cite[Chapter VIII]{Le}.  
\end{itemize} 

The following very useful lemma is due to Sz\'asz \cite{Sz}. 

\begin{lemma}[Sz\'asz]\label{lagg}
Let $H \subset \CC$ be an open half-plane with boundary containing the origin, and let $f(z)= b_Mz^M+b_{M+1}z^{M+1} +\cdots+b_Nz^N \in \CC[z]$, where $b_Mb_N \neq 0$. If $f(\zeta) \neq 0$ for all $\zeta \in H$,  then
 \begin{equation}\label{lag}
 |f(z)| \leq |b_M||z|^M\exp\left( \frac {|b_{M+1}|}{|b_{M}|}|z| + 3|z|^2 \frac {|b_{M+1}|^2}{|b_M|^2}+3|z|^2 \frac {|b_{M+2}|}{|b_M|}\right),  
\end{equation}
 for all $z \in \CC$. 
\end{lemma} 
\begin{remark}\label{S-pr}
The typical use of Lemma \ref{lagg} is as follows. Suppose that $\{ P_n(z) \}_{n=0}^\infty$ is a sequence of polynomials 
that are non-vanishing in  $H$, where $H$ is as in Lemma \ref{lagg}. Write 
$$
P_n(z)=\sum_{k=M}^{N_n} a_{n,k}z^k, 
$$
and let $\{a_k\}_{k=M}^\infty$ be a sequence of complex numbers with $a_M \neq 0$.  If $\lim_{n \rightarrow \infty} a_{n, k} =a_k $ for each $k \geq M$, then there is a subsequence of $\{ P_n(z) \}_{n=0}^\infty$ converging, uniformly on compact subsets of $\CC$, to the entire function $\sum_{k=M}^\infty a_k z^k$. This follows from Montel's theorem, since $\{ P_n(z) \}_{n=0}^\infty$ is locally uniformly bounded sequence by Lemma \ref{lagg}.  
\end{remark}

For a proof of the next lemma we refer to  \cite[Chapter VIII]{Le} or \cite[Theorem 12]{BB}. 

\begin{lemma}\label{lghu}
Let $\phi(z)= \sum_{k=0}^\infty a_k z^k/k!$ be a formal power series with complex coefficients, and let $H \subset \CC$ be an open half-plane with boundary containing the origin. Then  
$\phi(z)$ is an entire function which is the limit, uniformly on compact subsets of $\CC$, of polynomials that are non-vanishing in $H$  if and only if 
$$
\phi_n(z)= \sum_{k=0}^n \binom n k a_k z^{k}
$$
is either identically zero or non-vanishing in $H$, for each $n \in \NN$. 
\end{lemma}
\begin{theorem}\label{transhur}
Let $\mu=\{\mu_k\}_{k=0}^\infty$ be a sequence of complex numbers and let $\{\gamma_k\}_{k=0}^\infty$ be defined by 
\eqref{gammas}. Define a formal power series by 
$$
T_\mu(e^z)= \sum_{k=0}^\infty \frac {\gamma_k}{k!} z^k. 
$$
The following are equivalent. 
\begin{itemize}
\item[(i)]
For all polynomials $P(z)$  with only real and non-positive zeros,  the polynomial 
$T_\mu(P(z))$ is either identically zero or weakly Hurwitz stable; 
\item[(ii)] $T_\mu(e^z) \in \HH(\CC)\cup\{0\}$;
\item[(iii)] $T_\mu(\LP^+) \subseteq \HH(\CC)\cup\{0\}$. 
\end{itemize}
\end{theorem}

\begin{proof} Note that $\phi_n(z)$ is weakly Hurwitz stable if and only if $z^n\phi_n(1/z)$ is weakly Hurwitz stable. 
Combining Theorem \ref{GWS} and Lemma \ref{lghu} yields the equivalence of (i) and (ii). Clearly (iii) $\Rightarrow$ (ii).   Assume (i) and let $\phi \in \LP^+$. Then, by Lemma \ref{lghu}, $\phi_n(z)$ is a polynomial with only real and non-positive zeros (unless identically zero) for each $n \in \NN$. Thus $T_\mu(\phi_n(z/n))$ is weakly Hurwitz stable or identically zero for each $n \geq 1$. Note that 
$$
\lim_{n\rightarrow \infty} \binom n k  \frac {a_k} {n^k} = \frac {a_k} {k!}, 
$$
for each $k \in \NN$. 
By Remark \ref{S-pr}, there is a subsequence $\{n_j\}_{j=0}^\infty$ such that 
$$\lim_{j \rightarrow \infty} T_\mu(\phi_{n_j}(z/n_j))=T_\mu(\phi(z)),$$ where the convergence is uniform on each compact subset of $\CC$. Hence 
$T_\mu(\phi) \in \HH(\CC)\cup\{0\}$. 
\end{proof}

\begin{remark}
The characterization of the ``good'' sequences  $\{\mu_k\}_{k=0}^\infty$ in Theorem \ref{transhur} is in terms of the sequences $\{\gamma_k\}_{k=0}^\infty$. How do we translate between the two sequences? The answer is classical and is called the \emph{Chebyshev relation}, see \cite[p. 54]{Ri}:
$$
\gamma_k= \sum_{j=0}^{\lfloor k/2 \rfloor}\binom k j \mu_{k-2j}, \quad \mbox{ for all } k \in \NN
$$
if and only if 
$$
\mu_k= \sum_{j=0}^{\lfloor k/2 \rfloor}(-1)^j \frac k {k-j} \binom {k-j} j \gamma_{k-2j}, \quad \mbox{ for all } k \in \NN.
$$
\end{remark}

The following lemma follows easily from Remark \ref{S-pr} and the Hermite-Biehler Theorem. 
\begin{lemma}\label{trevod}
Let $\phi(z)$ be formal power series with real and non-negative coefficients. Then $\phi(z) \in \LP^+$ if and only if $\phi(z^2) \in \HH(\RR)$. 
\end{lemma} 

\begin{theorem}\label{trans-U}
Let $\alpha=\{\alpha_k\}_{k=0}^\infty$ be a sequence of real numbers, and let $n\in \NN$. The following are equivalent. 
\begin{itemize}
\item[(i)] $U_\alpha(\PP) \subseteq \PP\cup\{0\}$;
\item[(ii)]  $U_\alpha(e^z) \in  \LP^+\cup\{0\}$, that is, 
$$\sum_{k=0}^\infty \left( \sum_{j=0}^k \frac {\alpha_j}{(k+j)!(k-j)!} \right)z^k \in \LP^+\cup\{0\};$$
\item[(iii)] $U_\alpha(\LP^+) \subseteq \LP^+\cup\{0\}$. 
\end{itemize}
\end{theorem}
\begin{proof}
Let  $\mu=\{\alpha_0,0, \alpha_1,0, \alpha_2,\ldots\}$. Then $U_\alpha(\PP) \subseteq \PP\cup\{0\}$ if and only if $T_\mu(\PP) \subseteq \HH(\RR) \cup \{0\}$, by the Hermite--Biehler theorem. By Lemma \ref{trevod}, $U_\alpha(e^z) \in \LP^+$ if and only if $T_\mu(e^z) \in \HH(\RR)$, and $U_\alpha(\LP^+) \subseteq \LP^+\cup\{0\}$ if and only if $T_\mu(\LP^+) \subseteq \HH(\RR) \cup \{0\}$.

\end{proof}
The proof of the next theorem is almost identical to that of Theorem \ref{trans-U}. 
\begin{theorem}\label{trans-V}
Let $\alpha=\{\alpha_k\}_{k=0}^\infty$ be a sequence of real numbers, and let $n\in \NN$. The following are equivalent. 
\begin{itemize}
\item[(i)] $V_\alpha(\PP) \subseteq \PP\cup\{0\}$;
\item[(ii)]  $V_\alpha(e^z) \in \LP^+\cup\{0\}$, that is, 
$$\sum_{k=0}^\infty \left( \sum_{j=0}^k \frac {\alpha_j}{(k+1+j)!(k-j)!} \right)z^k \in \LP^+\cup\{0\};$$
\item[(iii)] $V_\alpha(\LP^+) \subseteq \LP^+\cup\{0\}$. 
\end{itemize}
\end{theorem}

\section{Applications and examples}

Let us apply Theorem \ref{trans-U} to a question posed by Fisk \cite{Fi}. For $r \in \NN$, let $S_r = U_\alpha$ where $\alpha_0=1, \alpha_r=-1$, and $\alpha_i=0$ for all $i \notin \{0,r\}$.  In other words 
$$
S_r\left( \sum_{i=0}^n a_i z^i\right) = \sum_{i=0}^n (a_i^2-a_{i-r}a_{i+r})z^i.
$$
Fisk asked whether $S_r(\PP) \subseteq \PP$ for all $r \in \NN$. We use Theorem \ref{trans-U} and the theory of multiplier sequences to obtain partial results on Fisk's question. 

A sequence of real numbers $\{\lambda_k\}_{k=0}^\infty$ is a {\em multiplier sequence} if for 
each polynomial $\sum_{k=0}^n a_k z^k$ with only real zeros, the polynomial
$
\sum_{k=0}^n \lambda_ka_k z^k
$
is either identically zero or has only real zeros. 

Multiplier sequences were characterized in a seminal paper by P\'olya and Schur \cite{PS}. That multiplier sequences preserve the Laguerre-P\'olya class follows easily from Remark \ref{S-pr} and Lemma \ref{lghu}, see also \cite[Chapter VIII]{Le}.  
\begin{theorem}[P\'olya and Schur]\label{ps}
Let $\{\lambda_k\}_{k=0}^{\infty}$ be a sequence of real numbers, and let 
$T: \RR[z] \rightarrow \RR[z]$ be the corresponding (diagonal) linear operator defined by $T(z^k)=\lambda_k z^k$, for all $k \in \NN$.
 Define $\Phi(z)=T(e^z)$ to be the formal power series 
$$
\Phi(z) = \sum_{k=0}^\infty \frac{\lambda_k}{k!}z^k.
$$
The following  are equivalent:
\begin{itemize}
\item[(i)]  $\{\lambda_k\}_{k=0}^{\infty}$ is a multiplier sequence;  
\item[(ii)] $T(\LP) \subseteq \LP\cup\{0\}$;
\item[(iii)] $\Phi(z)$ defines an entire 
function which is the limit, uniformly on compact sets, of 
polynomials with only real zeros of the same sign;
\item[(iv)] Either $\Phi(z)$ or $\Phi(-z)$ is an entire function 
that can be written as
$$
C z^n e^{az} \prod_{k=1}^\infty (1+ \alpha_k z),
$$ 
where $n \in \NN$, $C \in \RR$, $a,\alpha_k \geq 0$ for all $k \in \NN$ and 
$\sum_{k=1}^\infty \alpha_k < \infty$; 
\item[(v)] For all nonnegative integers $n$ the  polynomial 
$T[(1+z)^n]$ has only real zeros of the same sign.
\end{itemize}
\end{theorem}

\begin{proposition}\label{f1}
Let $r=0,1,2$ or $3$. Then $S_r(\PP) \subseteq \PP\cup \{0\}$. 
\end{proposition}
\begin{proof}
Fix $r \in \NN$, and let $S_r(e^z)=\sum_{k=0}^\infty a_{k, r}z^k$. Then 
$$
a_{k,r}= \frac 1 {k!(k+r)!}\Big( (k+1)\cdots (k+r)- k(k-1)\cdots (k-r+1)\Big). 
$$
In particular, 
$$
S_1(e^z)= \sum_{k =0}^\infty \frac 1 {k!(k+1)!}z^k.  
$$
For each $\mu >0$, the sequence $\{1/\Gamma(k+\mu)\}_{k=0}^\infty$ is a multiplier sequence, see \cite{CC3}. Thus $F_1(z) \in \LP^+$, by Theorem \ref{ps}. This verifies the case when $r=1$ by Theorem \ref{trans-U}.  

Next,  $a_{k,2}=(2+4k) / (k!(k+2)!)$.   
Both sequences $\{1/(k+2)!\}_{k=0}^\infty$ and $\{2+4k\}_{k=0}^\infty$ are multiplier sequences. Hence, so is $\{(2+4k)/(k+2)!\}_{k=0}^\infty$. The case $r=2$ now follows from Theorems \ref{ps} and \ref{trans-U}.  

Since $a_{k,3}=(6+9k+9k^2) / (k!(k+3)!)$, the case when $r=3$ follows from the fact that $\{6+9k+9k^2\}_{k=0}^\infty$ is a multiplier sequence. Indeed, 
$$
\sum_{k=0}^\infty \frac {6+9k+9k^2}{k!}z^k=(6+18z+9z^2)e^z,
$$
and the zeros of $6+18z+9z^2$ are negative. By Theorem \ref{ps}, $\{6+9k+9k^2\}_{k=0}^\infty$ is a multiplier sequence. 
\end{proof}

Similarly, let $S'_r = V_\alpha$, where $\alpha_0=1, \alpha_r=-1$, and $\alpha_i=0$ for all $i \notin \{0,r\}$.
\begin{proposition}
Let $r=0,1,2$ or $3$. Then $S'_r(\PP) \subseteq \PP\cup \{0\}$. 
\end{proposition}
\begin{proof}
Fix $r \in \NN$, and let $S'_r(e^z)=\sum_{k=0}^\infty b_{k, r}z^k$. Then 
$$
b_{k,r}= \frac 1 {k!(k+1+r)!}\Big( (k+2)\cdots (k+1+r)- k(k-1)\cdots (k-r+1)\Big). 
$$
The proof proceeds as the proof of Proposition \ref{f1}. For example, 
$$
b_{k,3} = \frac {12} {k!(k+4)!} (k^2+2k+2).
$$
The sequence $\{k^2+2k+2\}_{k=0}^\infty$ is a multiplier sequence since 
$$
\sum_{k=0}^\infty \frac {k^2+2k+2}{k!}z^k= (2+3z+z^2)e^z.
$$
\end{proof}
We conjecture that $S'_r(e^z) \in \LP^+$ for all $r \in \NN$. 

\section{Refined results on the location of zeros}
We provide here some general results on the effect on the zeros of polynomials under the transformations $T_\mu$, $U_\alpha$ and $V_\alpha$. 
\begin{theorem}\label{refined1}
Let $\mu$ be a sequence of complex numbers, and let $$P(z)= 1+a_1z+\cdots+a_nz^n=\prod_{j=1}^n(1+\rho_jz)$$ be a complex polynomial of degree $n$. Suppose that $K$ is a circular region containing no zeros of $P_{\mu,n}(z)$. We further require $K$ to be convex if $\mu_0=\gamma_0=0$.  
 If $\zeta$ is a non-zero complex number for which 
$$
\left\{\rho_i \zeta+ \frac 1 {\rho_i \zeta} : 1 \leq i \leq n \right\} \subset K,
$$
then $T_\mu(P)(\zeta) \neq 0$. 
\end{theorem}

\begin{proof}
Let $\zeta$ be as in the statement of the theorem, and suppose that  $T_\mu(P)(\zeta) = 0$. Since 
$$
T_\mu(P)(\zeta) = W_{\mu,n}(\rho_1\zeta, \ldots, \rho_n \zeta)= a_n\zeta^n \sum_{k=0}^n \gamma_k e_{n-k}\left(\rho_1\zeta+\frac 1 {\rho_1\zeta}, \ldots, \rho_n\zeta+\frac 1 {\rho_n\zeta}\right), 
$$
there is, by Theorem \ref{GWS}, a $\xi \in K$ such that 
$$
0= \sum_{k=0}^n \gamma_k e_{n-k}\left(\xi, \ldots, \xi \right)= \sum_{k=0}^n \gamma_k \binom n k \xi^{n-k}=P_{\mu,n}(\xi), 
 $$
 which contradicts the assumptions on $K$. 
\end{proof}

 For $0<\theta<2\pi$, let
$S_\theta= \{ re^{i\phi}: |\pi-\phi|<\theta \mbox{ and } r>0\}$ be the sector centered on the negative real axis, and  that opens an angle $2\theta$. 
\begin{theorem}\label{refined2}
Let $\alpha=\{\alpha_k\}_{k=0}^\infty$ be a sequence of real number such that $$U_\alpha((1+z)^n) \in \PP_n.$$  Suppose that $P(z)= \sum_{k=0}^n a_k z^k$ has zeros only in $S_\theta$, where $0 \leq \theta <\pi/2$. Then $U_\alpha(P(z)) \equiv 0$, or all zeros of 
$U_\alpha(P(z))$ are in $S_{2\theta}$. 
\end{theorem}
\begin{proof}
Suppose that $P(z)= \sum_{k=0}^n a_k z^k$ has zeros only in $S_\theta$. Write $P(z)=C \prod_{j=1}^n (1+\rho_jz)$. Then $|\arg(\rho_j)|< \theta$ for all $1\leq j \leq n$. If $|\arg(\zeta)|< \pi/2-\theta$, then 
$$
\left\{\rho_j \zeta+ \frac 1 {\rho_j \zeta} : 1 \leq j \leq n \right\} \subset \{z  \in \CC : \Re(z)>0\}.
$$
By Theorem \ref{refined1}, $T_\mu(P)(\zeta) \neq 0$, where $\mu=\{\alpha_0,0,\alpha_1,0,\ldots\}$. Thus $U_\alpha(z)= T_\mu(\sqrt{z}) \neq 0$, whenever $|\arg(z)|< \pi-2\theta$. 
\end{proof}
The proof of the corresponding theorem for $V_\alpha$ is almost identical. 
\begin{theorem}\label{refined3}
Let $\alpha=\{\alpha_k\}_{k=0}^\infty$ be a sequence of real numbers such that $$V_\alpha((1+z)^n) \in \PP_n.$$  Suppose that $P(z)= \sum_{k=0}^n a_k z^k$ has zeros only in $S_\theta$, where $0 \leq \theta <\pi/2$. Then $V_\alpha(P(z)) \equiv 0$, or all zeros of 
$V_\alpha(P(z))$ are in $S_{2\theta}$. 
\end{theorem}

\section{Iterated Tur\'an inequalities and the Boros--Moll conjecture}\label{fin}
For Taylor coefficients of  functions in $\LP^+$, inequalities stronger than log-concavity hold. Namely the \emph{Tur\'an inequalities}: If $\sum_{k=0}^\infty \gamma_k z^k/k! \in \LP^+$, then 
the sequence $\{\gamma_k\}_{k=0}^\infty$ is log-concave. Craven and Csordas \cite{CC2} studied \emph{iterated Tur\'an inequalities}. Define a transformation, $\TT$, on infinite sequences as follows. If $\{\gamma_k\}_{k=0}^\infty$ is a sequence, let $\TT(\{\gamma_k\})=\{\mu_k\}_{k=0}^\infty$, where 
$\mu_k=\gamma_{k+1}^2-\gamma_k\gamma_{k+2}$. Note the shift of indices. Craven and Csordas stated the following problem.
\begin{problem}\label{it-turan}
Let  $\sum_{k=0}^\infty \gamma_k z^k/k! \in \LP^+$.  Is $\TT^i(\{\gamma_k\})$ a non-negative sequence for all 
$i \in \NN$?
\end{problem}
Craven and Csordas \cite{CC2} proved that $\TT^2(\{\gamma_k\})$ is non-negative if $\sum_{k=0}^\infty \gamma_k z^k/k! \in \LP^+$, and that $\TT^3(\{\gamma_k\})$ is non-negative if $\sum_{k=0}^\infty \gamma_k z^k/k! \in \LP^+$ and $\gamma_0=\gamma_1=0$. The second result can be stated as follows: If $\sum_{k=0}^\infty \gamma_k z^k/(k+2)! \in \LP^+$, then $\{\gamma_k\}_{k=0}^\infty$ is $3$-fold log-concave. In \cite{CC1} they posed the  following  problem. 
\begin{problem}\label{ccprob}
Characterize the sequences $\{\gamma_k\}_{k=0}^\infty$ such that 
$$
\sum_{k=0}^\infty \frac {\gamma_k}{k!}z^k \in \LP^+ \quad \mbox{ and } \quad  \sum_{k=0}^\infty \frac {t_k}{k!}z^k \in \LP^+ , 
$$
where $\{t_k\}_{k=0}^\infty=\TT(\{\gamma_k\}_{k=0}^\infty)$.
\end{problem}
Theorem \ref{trans-U} provides a large class of entire functions for which Problem \ref{ccprob} holds. 
\begin{proposition}
If $\sum_{k=0}^\infty\gamma_k z^k \in \LP^+$, then $\{\gamma_k\}_{k=0}^\infty$ satisfies both conditions in Problem \ref{ccprob}. 
\end{proposition}
\begin{proof}
Since $\{1/k!\}_{k=0}^\infty$ is a multiplier sequence, Theorem \ref{ps} implies 
$\sum_{k=0}^\infty  {\gamma_k}z^k/{k!} \in \LP^+$. We claim that $\{1/(k-1)!\}_{k=0}^\infty$, where $1/(-1)!:=0$, is a multiplier sequence. Indeed 
$$
\sum_{k=0}^\infty \frac 1 {(k-1)!k!}z^k = z \sum_{k=0}^\infty \frac 1 {k!(k+1)!}z^k \in \LP^+. 
$$
By Theorem \ref{trans-U}, 
$$
\sum_{k=0}^\infty (\gamma_k^2-\gamma_{k-1}\gamma_{k+1})z^k \in \LP^+,
$$
and by Theorem \ref{ps}, 
$$
\sum_{k=0}^\infty \frac{\gamma_k^2-\gamma_{k-1}\gamma_{k+1}} {(k-1)!}z^k = z\sum_{k=0}^\infty \frac {t_k}{k!}z^k \ \in \LP^+. 
$$

\end{proof}

Let us describe the initial conjecture that motivated Boros and Moll to study infinitely log-concave sequences. 
For $\ell, m \in  \NN$ with $\ell \leq m$, let 
\begin{equation}\label{d-def}
d_\ell(m)= 2^{-2m}\sum_{k=\ell}^m2^k\binom {2m-2k}{m-k}\binom {m+k} m \binom k \ell.
\end{equation}
It is not trivial (at least without the use of computers) to prove that $d_\ell(m)$ is the $\ell$th Taylor coefficient of the polynomial, defined for $a>-1$, by 
$$
P_m(a)= \frac{2^{m+3/2}(a+1)^{m+1/2}}{\pi} \int_0^\infty \frac 1 {(x^4+2ax^2+1)^{m+1}}dx.
$$

Based on computer experiments, Boros and Moll made the following conjecture, see \cite{BM}. 
\begin{conjecture}\label{orig-con}
For each $m \in \NN$, the sequence $\{d_\ell(m)\}_{\ell=0}^m$ is infinitely log-concave. 
\end{conjecture}
Kauers and Paule \cite{KP} were able to prove log-concavity of $\{d_\ell(m)\}_{\ell=0}^m$, using computer algebra. 
We make the following conjecture.
\begin{conjecture}\label{fact0}
For each $m \in \NN$, the polynomial 
$$
Q_m(z)= \sum_{\ell=0}^m \frac{d_\ell(m)}{\ell!}z^\ell
$$
has only real zeros. 
\end{conjecture}
We also make a stronger conjecture. 
\begin{conjecture}\label{fact2}
For each $m \in \NN$, the polynomial 
$$
R_m(z)= \sum_{\ell=0}^m \frac{d_\ell(m)}{(\ell+2)!}z^\ell
$$
has only real zeros. 
\end{conjecture}
Note that $Q_m(z)= (d^2/dz^2)(z^2R_m(z))$, so Conjecture \ref{fact0} is stronger than  Conjecture \ref{fact2}.
The relevance of these conjectures stems from the results of Craven and Csordas on Problem \ref{it-turan}. If Conjecture \ref{fact0} is true, then $\{d_\ell(m)\}_{\ell=0}^m$ is $2$-fold log-concave. If Conjecture \ref{fact2} is true, then $\{d_\ell(m)\}_{\ell=0}^m$ is $3$-fold log-concave. \\[2ex]
{\bf Acknowledgments.} I thank the anonymous referee for carefully reading the paper.

\end{document}